\documentclass[10pt]{article}
\usepackage{amsmath,amssymb,indentfirst,bm}
\usepackage{graphicx, epsfig, url}
\usepackage[authoryear]{natbib}
\usepackage{amsthm, verbatim}
\usepackage{setspace}
\newtheorem{definition}{Definition}[section]
\newtheorem{theorem}{Theorem}[section]
\newtheorem{example}{Example}[section]
\newtheorem{lemma}{Lemma}[section]

\title{A classical measure of evidence for general null hypotheses}
\author{Alexandre Galv\~ao Patriota\\
     {\it \footnotesize Departamento de Estat\'istica, IME,
     Universidade de S\~ao Paulo}
     \vspace{-0.2cm}\\
     {\it \footnotesize Rua do Mat\~ao, 1010, S\~ao Paulo/SP, 05508-090, Brazil}
     \vspace{-0.2cm}\\
     {\footnotesize email: {\tt patriota@ime.usp.br}} \\ \\     
}
\date{}

\begin{document}
\maketitle
\onehalfspace
\begin{abstract}
In science, the most widespread statistical quantities are perhaps $p$-values. A typical advice is to reject the null hypothesis $H_0$ if the corresponding $p$-value is sufficiently small (usually smaller than $0.05$). Many criticisms regarding $p$-values have arisen in the scientific literature. The main issue is that in general optimal $p$-values (based on likelihood ratio statistics) are not measures of evidence over the parameter space $\Theta$. Here, we propose an \emph{objective}  measure of evidence for very general null hypotheses that satisfies logical requirements (i.e., operations on the subsets of $\Theta$) that are not met by $p$-values (e.g., it is a possibility measure). We study the proposed measure in the light of the abstract belief calculus formalism and we conclude that it can be used to establish objective states of belief on the subsets of $\Theta$. Based on its properties, we strongly recommend this measure as an additional summary of significance tests. At the end of the paper we give a short listing of possible open problems.

\vspace{0.2cm}
\noindent{\it Keywords}: Abstract belief calculus, evidence measure, likelihood-based confidence, nested hypothesis, $p$-value, possibility measure, significance test
\end{abstract}

\section{Introduction}\label{sec1}

Tests of significance are subjects of intense debate and discussion among many statisticians \citep{Kemp, Cox1977, BergerSellke, Aitkin,Schervish,Royall1997,  MayoCox}, scientists in general \citep{DuboisPrade1990, DarwicheGinsberg, Friedman, Wagenmakers} and philosophers of science \citep{Stern2003,Mayo2004}. In this paper, we discuss some limitations of $p$-values (which is a well-explored territory) and we propose an alternative measure to establish objective states of belief on the subsets of the full parameter space $\Theta$. Currently, many scholars have been studying the controversies and limitations of $p$-values \citep[see, for instance,][]{MayoCox, MayoSpanos2006, Wagenmakers, PereiraSternWechsler,Rice2010, Grendar2012, Diniz2012} and others have proposed some alternatives \citep{Zhang2009, Bickel2012}. In this paper, besides proposing an objective measure of evidence, we also provide a connection with the abstract belief calculus (ABC) proposed by \cite{DarwicheGinsberg}, which certifies the status of ``objective state of belief'' for our proposal, see Section \ref{ABCformalis} for specific details.

A procedure that measures the consistency of an observed data $x$ (the capital letter $X$ denotes the random quantity) with a null hypothesis $H_0: \theta \in \Theta_0$ is known as a significance test \citep{Kemp,Cox1977}. According to \cite{MayoCox}, to do this in the frequentist paradigm, we may find a function $t= t(x)$ called test statistic such that: (1) the larger the value of $t$ the more inconsistent are the data with $H_0$ and (2) the random variable $T = t(X)$ has known probability distributions (at least asymptotically) under $H_0$. The $p$-value related to the statistic $T$ \citep[the observed level of significance][]{Cox1977} is the probability of an unobserved $T$ to be, at least, as extreme as the observed $t$, under $H_0$. In the statistical literature is common  to informally define p-values as 
\begin{equation}\label{p-value-Inf}
p(\Theta_0) = P(T > t; \mbox{under } H_0),
\end{equation}
see, for instance, \cite{MayoCox}. Notice that, small values of $p$ indicate a discordance of the data probabilistic model from that specified in $H_0$. It is common practice  to set in advance a threshold value $\alpha$ to reject $H_0$ if and only if $p \leq \alpha$. The informal definition in Equation (\ref{p-value-Inf}) leads to mistaken interpretations and can feed many controversies, since one is driven to think that the measure $P$ and the statistic $T$ does not depend upon the null set $\Theta_0$. Some of the critics against the use of $p$-values follow. \cite{PereiraWechsler} point out some problems when the statistic $T$ does not consider the alternative hypothesis. \cite{Schervish} had argued that $p$-values as measures of evidence for hypotheses has serious logical flaws. \cite{BergerSellke} argue that $p$-values can be highly misleading measures of the evidence provided by the data against the null hypothesis.

In this paper, we provide a formal definition of $p$-values and present two examples where conflicting conclusions arise if $p$-values are used to take decisions regarding the inadequacy of a hypothesis. Then, we propose a measure of evidence for general null hypotheses that is free of those conflicts and has some important philosophical implications in the frequentist paradigm which will be detailed in future works.

Here, the null hypothesis is defined in a parametric context, let $\theta \in \Theta \subseteq \mathbb{R}^k$ be the model parameter, the null hypothesis is defined as $H_0: \theta \in \Theta_0$. One interpretation is: ``$H_0$ is considered true when the true unknown value of the parameter vector lies in the subset $\Theta_0 \subseteq \Theta$''. A second interpretation reads: ``$H_0$ is considered true when the probability measures indexed by the elements of $\Theta_0 \subseteq \Theta$ explain more efficiently the random events than the probability measures indexed by the elements of $\Theta_0^c = \Theta - \Theta_0$'', where \emph{more efficiently} is relative to certain criteria. 
Basically, a hypothesis test attempts to reduce a family of possible measures that governs the data behavior, say $\mathcal{P}_X = \{\mu_\theta: \ \theta \in \Theta\}$, to a more restricted one, say $\mathcal{P}_X^0 = \{\mu_\theta : \ \theta \in \Theta_0\}$. 

In the following we define a $p$-value precisely, then we can properly understand some of its features. Let $\mathcal{P}_X = \{\mu_\theta; \ \theta \in \Theta\}$ be a family of probability measures induced by the random sample $X$. In optimal tests, the reader should notice that the statistic $T$ for testing $H_0$  depends, in general, on the null set $\Theta_0$, thus it should be read as $T_{\Theta_0}$ instead of $T$. In order to avoid further misunderstandings, we decided to take into account this index from now on. As $T_{\Theta_0}$ is a function of the random sample we also have an induced family of probability measures $\mathcal{P}_{T_{\Theta_0}} = \{P_{\theta, \Theta_0}; \theta \in \Theta\}$, where $P_{\theta,{\Theta_0}} \equiv \mu_\theta T_{\Theta_0}^{-1}$ is a measure that depends on the null set $\Theta_0$ and the parameter vector $\theta$. The informal statement ``under $H_0$'' means  a subfamily of probability measures restricted to the null set, namely $\mathcal{P}_{T_{\Theta_0}}^0 = \{P_{\theta, \Theta_0}; \theta \in \Theta_0\}$. Then, it is possible to define many p-values as we can see below 
\[p_\theta(\Theta_0) = P_{\theta, \Theta_0}(T_{\Theta_0}>t), \quad \mbox{ for } \theta \in \Theta_0\]
and the most conservative p-value over $\Theta_0$ can be defined as
\[ p(\Theta_0)= \sup_{\theta \in \Theta_0} p_\theta(\Theta_0). \]

When $P_{\theta, \Theta_0} \equiv P_{\Theta_0}$ for all $\theta \in \Theta_0$, i.e., the statistic $T_{\Theta_0}$ is ancillary to the family $\mathcal{P}_T^0$, then all p-values are equal: $p_{\theta^{'}}(\Theta_0) = p_{\theta^{''}}(\Theta_0)$ for all $\theta^{'}, \theta^{''} \in \Theta_0$ (see Examples \ref{Ex-normal1} and \ref{reg}). If  $P_{\theta, \Theta_0} \equiv P_{\Theta_0}$  for all $\theta \in \Theta_0$ happen asymptotically we say that $T_{\Theta_0}$ is asymptotically ancillary to $\mathcal{P}_T^0$. As it is virtually impracticable in complex problems to find exact probability measures $P_{\theta, \Theta_0}$ for all $\theta \in \Theta_0$ we can use the asymptotic distribution, that is, $P_{\Theta_0}$ is the probability measure correspondent to the asymptotic distribution of $T_{\Theta_0}$. 

There are many ways to find a test statistic $T_{\Theta_0}$, it essentially depends on the topologies of $\Theta_0$ and $\Theta$. When $\Theta_0$ and its complement have one element each, the Neyman-Pearson Lemma provides the most powerful test (which is the likelihood ratio statistic) for any pre-fixed significance value. Naturally, we can use this statistic to compute a $p$-value. For the general case, the generalization of likelihood ratio statistic (which will be called only by likelihood ratio statistic) is given by
\[
\lambda_{\Theta_0}(x) = \frac{\sup_{\theta \in \Theta_{0}} L(\theta, x)}{\sup_{\theta \in \Theta} L(\theta, x)}
\] where $L(\theta, x)$ is the likelihood function. The testing statistic can be defined as $T_{\Theta_0} = -2\log(\lambda_{\Theta_0}(X))$, since this has important asymptotic properties as we will see below. Observe that the likelihood ratio statistic does take into account the alternative hypothesis (since $\Theta = \Theta_0 \cup \Theta_1$  with $\Theta_1$ being the parameter space defined in the alternative hypothesis). \cite{Mudholkar} studied optimal $p$-values for very general null hypotheses (considering both one and two-sided null hypotheses) that take into account the corresponding alternative hypotheses, some of these optimal $p$-values are computed by using likelihood ratio statistics. The reader should notice that the likelihood ratio statistic is a generalization for uniformly most powerful tests \citep[see][]{BIRKES} under general hypothesis testing. For general linear hypothesis, $H_0: C \theta  = d$, we can also resort to a Wald-type statistic
\[
W(x) = (C\widehat{\theta} - d)^\top [CAC^\top]^{-1}(C\widehat{\theta} - d)
\] with $\widehat{\theta}$ being a consistent estimator that, under $H_0$, is (asymptotically) normally distributed and $A$ its (asymptotic) covariance-variance matrix computed at $\widehat{\theta}$. These two statistics share many important properties and are widely used in actual problems. Suppose that $X = (X_1, \ldots, X_n)$ is an independent and identically distributed (iid) random sample, under some regular conditions on $L(\theta, X)$ and when $\Theta_0$ is a smooth (semi)algebraic manifold  with $\mbox{dim}(\Theta_0)< \mbox{dim}(\Theta)$, it is well known that, under $H_0$, $T_{\Theta_0}(X) = -2\log(\lambda_{\Theta_0}(X))$ converges in distribution to a chisquare distribution with $r$ degrees of freedom (from now on, it is denoted just by $\chi^2_r$), where $r= \mbox{dim}(\Theta) - \mbox{dim}(\Theta_0)$ is the co-dimension of $\Theta_0$. The asymptotic distribution of $W(X)$ is a chisquare with rank-of-$C$ degrees of freedom, which is the very same of the likelihood ratio statistics for linear general null hypotheses. We can also mention the Score test statistics that, under appropriated conditions, has asymptotically the same distribution as the two previous statistics. That is, different $p$-values can be computed for the same problem of hypothesis testing by using different procedures. In this paper we shall only use the procedure based on the (generalized) likelihood ratio statistics, since it has optimal asymptotic properties \citep[see][]{Bahadur}. From now on, ``asymptotic $p$-values'' means $p$-values computed by using the asymptotic distribution of the test statistic.

Sometimes practitioners have to test a complicated hypothesis $H_{01}$. By reasons of easiness of computations, instead of testing $H_{01}$, they may think of testing another auxiliar hypothesis $H_{02}$ such that if $H_{02}$ is false then $H_{01}$ is also false. This procedure is used routinely in medicine and health fields in general, e.g., in genetic studies one of the interests is to test genotype frequencies between two groups \citep{Izbicki}. In this example, we know that $H_{01}:$ ``homogeneity of genotype frequencies between the groups'' implies $H_{02}:$ ``homogeneity of allelic frequencies between the groups''. By using mild logical requirements, if we find evidence against $H_{02}$ we expect to claim evidence against $H_{01}$. However, as it is widely known, $p$-values do not follow this logical reasoning. 

Let $H_{01}: \theta \in \Theta_{01}$ and $H_{02}: \theta \in \Theta_{02}$ be two null hypotheses such that $\Theta_{01}\subseteq \Theta_{02}$, i.e., $H_{01}$ is nested within $H_{02}$. It is expected by the logical reasoning to find more evidence against $H_{01}$ than $H_{02}$ for the same observed data $X = x$. In other words, if $p_i$ is the asymptotic $p$-value computed under $H_{0i}$, for $i=1,2$, respectively, then we expect to observe $p_1 < p_2$.
 However, if dimensions of the spaces described in these nested hypotheses are different, then their respective asymptotic $p$-values will be computed under different metrics and therefore inverted conclusions may occur, i.e., more disagreement with $H_{02}$ than $H_{01}$ (i.e., $p_2<p_1$). That is to say, for a given data and a preassigned $\alpha$, it may happen $p_1> \alpha$ and $p_2 < \alpha$. One, therefore, may be confronted at the same time with ``evidence'' to reject $H_{02}$ and without ``evidence'' to reject $H_{01}$. Of course, the problem here is not with the approximation for the $p$-values, computed by using limiting reference distribution, the problem also happens with the exact ones. The example below shows the above considerations for exact $p$-values in a multiparametric scenario.

\begin{example}\label{Ex-normal1} Consider an independent and identically distributed random sample $X= (X_1,$ $\ldots, X_n)$ where $X_1 \sim N_2(\mu, I)$ with $\mu = (\mu_1, \mu_2)^\top$ and $I$ is a $(2 \times 2)$ identity matrix. The full parameter space is $\Theta = \{(\mu_1, \mu_2): \ \mu_1, \mu_2 \in \mathbb{R}\} = \mathbb{R}^2$. For this example we consider two particular hypotheses. Firstly, suppose that we want to test  $H_{01}: \theta \in \Theta_{01}$, where $\Theta_{01} = \{(0,0)\}$, then the likelihood ratio statistic is \[\lambda_{\Theta_{01}}(X) = \frac{\sup_{\theta \in \Theta_{01}} L(\theta, X)}{\sup_{\theta \in \Theta} L(\theta, X)} = \exp\bigg(-\frac{n}{2} \bar{X}^\top\bar{X}\bigg),\] where $\bar{X}$ is the sample mean. Taking $T_{\Theta_{01}}(X) = -2 \log(\lambda_{\Theta_{01}}(X))$ we know that, under $H_{01}$, $T_{\Theta_{01}} \sim \chi^2_2$. Secondly, suppose that the null hypothesis is $H_{02}: \theta \in \Theta_{02}$, where $\Theta_{02} \equiv \{(\mu_1, \mu_2): \ \mu_1 = \mu_2, \ \mu_1, \mu_2 \in \mathbb{R}\}$,  the likelihood ratio statistic is \[\lambda_{\Theta_{02}}(X) = \frac{\sup_{\theta \in \Theta_{02}} L(\theta, X)}{\sup_{\theta \in \Theta} L(\theta, X)} = \exp\bigg(-\frac{n}{2} \bar{X}^\top(I - \frac{1}{2}ll^\top)\bar{X}\bigg),\] where $l = (1,1)^\top$. Taking $T_{\Theta_{02}}(X) = -2 \log(\lambda_2(X))$ it is possible to show that, under $H_{02}$, $T_{\Theta_{02}} \sim \chi^2_1$. Notice that, in this example, the Wald statistics for these two null hypotheses $H_{01}$ and $H_{02}$ are equal to $T_{\Theta_{01}}$ and $T_{\Theta_{02}}$, respectively. Assume that the sample size is $n=100$ and the observed sample mean is $\bar{x} = (0.14, -0.16)^\top$, then $T_{\Theta_{01}}(x) = 4.52$ (with $p$-value $p_1 = 0.10$) and $T_{\Theta_{02}}(x) = 4.5$ (with $p$-value $p_2 = 0.03$). These $p$-values showed evidence against $\mu_1 = \mu_2$, but not against $\mu_1 = \mu_2 = 0$. However, if we reject that $\mu_1 = \mu_2$ we should technically reject that $\mu_1 = \mu_2 = 0$ (using the very same data). 
\end{example}

This issue does not happen only with the likelihood ratio statistic, it happens with many other classical test statistics (score and others) that consider how data should behave under $H_0$. As $p$-values are just probabilities to find unobserved statistics, at least, as large as the observed ones, the conflicting conclusion presented in the above example is not a logical contradiction of the frequentist method. \emph{This issue happens because a $p$-value was not designed to be a measure of evidence over subsets of $\Theta$}. We must say that $p$-values do exactly the job they were defined to do. However, in the practical scientific world, researches use $p$-values to take decisions and, hence, they eventually may face some problems with consistency of conclusions. P-values must therefore be used with caution when taking decisions about a null hypothesis. 

The example below presents a data set which produces surprising conclusions for regression models. 

\begin{example}\label{reg} Consider a linear model: $y = xb + e$, where $b = (b_1, b_2)^\top$ is a vector formed by two regression parameters, $x = (x_1, x_2)$ is an $(n \times 2)$ matrix of covariates and $e \sim N_{n}(0, I_n)$ with $I_n$ the $n\times n$ identity matrix. It is usual to verify if each component of $b$ is equal to zero and to remove from the model the non-significant parameters. The majority of statistical routines present the $p$-values for $H_{0i}: b_i = 0$, say $p_i$, for $i=1,2$. However, sometimes both $p$-values are greater than $\alpha$ and there exists a joint effect that cannot be discarded. As these hypotheses include a more restricted one, $H_{03}: b = 0$, it is of general advice to reject $H_{03}$ only if the $p$-value $p_3$ is smaller  than $\alpha$ (this decision obeys the logical reasoning). We expect to observe more evidence against $H_{03}$ than either $H_{01}$ and $H_{02}$. In fact, almost always the $p$-value $p_3$ is smaller than both $p_1$ and $p_2$, as expected. However, as we shall see below, an inversion of conclusions may occur. To see that, let us present the main ingredients. The maximum likelihood estimator of $b$ is $\widehat{b} = (x^\top x)^{-1} x^\top y$ and the likelihood ratio statistics for testing $H_{01}$, $H_{02}$ and $H_{03}$ are respectively
\[\lambda_{\Theta_{0i}}(y,x) = \exp\bigg(-\frac{1}{2} y^\top\bigg(x(x^\top x)^{-1}x^\top - \dot{x}_i(\dot{x}_i^\top \dot{x}_i)^{-1} \dot{x}_i^\top\bigg)y\bigg)\] for $i = 1,2$ and
\[ \lambda_{\Theta_{03}}(y,x) = \exp\bigg(-\frac{1}{2} y^\top x(x^\top x)^{-1}x^\top y\bigg),\] where $\dot{x}_1 = x_2$ and $\dot{x}_2 = x_1$.
It can be showed that $T_{\Theta_{0i}}(Y,x) = -2 \log(\lambda_{\Theta_{0i}}(Y,x)) \sim \chi^2_{s_i}$ for all $i=1,2,3$, where, for this example, $s_1 = s_2 = 1$ and $s_3 = 2$. Again, although $T_{\Theta_{03}}  > T_{\Theta_{0i}}$, for $i=1,2$, the metrics to compute the $p$-values are different and odd behavior may arise as we notice in the following data,
\[
\begin{array}{rrrrrrrrrrr}\hline
 y & -1.29& 1.09& -0.16&  0.44& -0.22& -1.85& 0.91& 0.54& 0.06&  0.37\\
x_1 & 3.00& 8.00 & 5.00&  9.00 &10.00 & 1.00 &6.00& 9.00 &6.00 & 5.00\\
x_2 & 9.00& 7.00 & 7.00& 10.00 & 7.00 & 8.00 &6.00& 6.00 &3.00 & 2.00\\\hline
\end{array}
\]
Here, the observed three statistics are $t_{\Theta_{01}} = 4.48$ (with $p$-value $p_1 = 0.03$), $t_{\Theta_{02}} = 4.00$ (with $p$-value $p_2 = 0.045$) and $t_{\Theta_{03}} = 4.59$ (with $p$-value $p_3 = 0.10$). For these data, we have problems with the conclusion, since we expected to have much more evidence against $H_{03}$ than $H_{01}$ and $H_{02}$. Notice that, $p_3/p_1 \sim 3.3$ and $p_3/p_2 \sim 2.2$.
\end{example}



Many other examples for higher dimensions can be built on, but we think that these two instances are sufficient to illustrate the weakness of $p$-values when it comes to decide acceptance or rejection of specific hypotheses, for other examples we refer the reader to \cite{Schervish}. In the above examples, we used the very same procedure to test both hypotheses  $H_{01}$ and $H_{02}$ (i.e., likelihood ratio statistics). Some scientists and practitioners would become confused with these results and it would be very difficult explain to them the reason for that. We believe that the development of a true measure of evidence for null hypotheses that does not have these problems might be welcome by the scientific community.

In summary: in usual frequentist significance tests, a general method of computing test statistics can be used (likelihood ratio statistics, Wald-type statistics, Score statistic and so forth). The distribution of the chosen test statistic depends on the null hypothesis and this leads to different metrics in the computation of $p$-values (this is the major factor that gives the basis for the frequentist interpretations of $p$-values). As each of these metrics depend on the dimension of the respective null hypotheses, conflicting conclusions may arise for nested hypotheses. In the next section, we present a new measure that can be regarded as a measure of evidence for null hypotheses without committing any logical contradictions.

This paper unfolds as follows. In Section \ref{sec-Evid} we present a definition of evidence measure and propose a frequentist version of this measure. Some of its properties are presented in Section \ref{prop}. A connection with the abstract belief calculus is showed in Section \ref{ABCformalis}. Examples are offered in Section \ref{Exam}. Finally, in Section \ref{Disc} we discuss the main results and present some final remarks. 

\section{An evidence measure for null hypotheses}\label{sec-Evid}

In this section we define a very general procedure to compute a measure of evidence for $H_0$. The concept of evidence was discussed by \cite{Good1983} in a great philosophical detail. We also refer the reader to \cite{Royall1997} and its review \cite{Vieland} for relevant arguments to develop new methods of measuring evidence.  As in the previous section, $H_0: \theta \in \Theta_0$ is the null hypothesis, where $\Theta_0 \subseteq \Theta \subseteq \mathbb{R}^k$ is a smooth manifold. Below, we define what we mean by an objective evidence measure.

\begin{definition}\label{Evid} Let $\mathcal{X} \subseteq \mathbb{R}^n$ be the sample space and $\mathbb{P}(\Theta)$ the power set of $\Theta$. A function $s : \mathcal{X}\times \mathbb{P}(\Theta) \to [0,1]$ is a measure (we shall write just $s(\Theta_0) \equiv s(X, \Theta_0)$ for shortness of notation, where $X \in \mathcal{X}$ is the data) of evidence of null hypotheses if the following items hold
\begin{enumerate}
\item[1.] $s(\varnothing) = 0$ and $s(\Theta) = 1$,
\item[2.] For any two null hypotheses $H_{01}: \theta \in \Theta_{01}$ and $H_{02}: \theta \in \Theta_{02}$, such that $\Theta_{01} \subseteq \Theta_{02}$, we must have $s(\Theta_{01}) \leq s(\Theta_{02})$,
\end{enumerate}
\end{definition}

The above definition is the least we would expect from a coherent measure of evidence. Items 1 and 2 of Definition \ref{Evid} describe a plausibility measure \citep{Friedman}, which generalizes probability measures. As showed in the previous section, $p$-values are not even plausibility measures on $\Theta$, since Condition 2 of Definition \ref{Evid} is not satisfied. Therefore, they cannot be regarded as measures of evidence. Bayes factors are also not plausibility measures on $\Theta$, i.e., Condition 2 fails to be held,  \citep[see][]{LavineSchervish, Bickel2012}. As pointed out by a referee, based on Definition \ref{Evid}, many measures can be qualified as a measure of evidence, even posterior probabilities. Here we restrict ourselves to be objective in the sense that no prior distributions neither over $\Theta$ nor for $H_0$ are specified, i.e., that the strength of evidence does not vary from one researcher to another \citep{Bickel2012}. Moreover, the proposed measure of evidence should be invariant under reparametrizations, this is an important feature to guarantee that the measure of evidence is not dependent upon different parametrizations of the model. In order to find a purely objective measure of evidence with these characteristics, without prior distributions neither over $\Theta$ nor $H_0$, we define a likelihood-based confidence region by

\begin{definition}\label{conf} A likelihood-based confidence region with level $\alpha$ is
\[
{\Lambda}_\alpha = \{\theta \in \Theta: T_{\theta} \leq F_\alpha\},
\] where $T_{\theta}= 2(\ell(\widehat{\theta}) - \ell(\theta))$, $\widehat{\theta}= \arg \sup_{\theta \in \Theta} \ell(\theta)$ is the maximum likelihood estimator, $\ell$ is the log-likelihood function and $F_{\alpha}$ is an $1-\alpha$ quantile computed from a cumulative distribution function $F$, i.e., $F(F_{\alpha}) = 1-\alpha$. Here, $F$ is (an approximation for) the cumulative distribution function  of the random variable $ T_{\theta_0} $ that does not depend on $\theta_0$, where $\theta_0$ is the true value. 
\end{definition}

Notice that, the cumulative distribution of  $ T_{\theta_0} $ is given by
\[F(t) = P_{\theta_0}(T_{\theta_0} \leq t).\]
Here, we assume that it is free of $\theta_0$ (otherwise, consider the asymptotic approximation).

As aforementioned, some of the optimal $p$-values studied by \cite{Mudholkar} are explicitly based on likelihood ratio statistics and this motives the use of the likelihood-based confidence region to build our measure of evidence. Moreover, as pointed out by a referee, \cite{Sprott} provides examples for confidence regions not based on likelihood functions that produce absurd regions. They are strong cases for using likelihood-confidence regions, i.e., confidence regions that are based on likelihood functions. 

Below we define an evidence measure for the null hypothesis $H_0$.

\begin{definition}\label{s-value} Let $\Lambda_\alpha$ be the likelihood-based confidence region. The evidence measure for the null hypothesis $H_0: \theta \in \Theta_0$ is the function $s : \mathcal{X}\times \mathbb{P}(\Theta) \to [0,1]$ such that \[s = s(\Theta_0) \equiv s(X,\Theta_0) = \max\{0,\sup\{\alpha \in (0,1): \ \Lambda_\alpha \cap  \Theta_0 \neq \varnothing\}\}.\] We shall call $s$-value for short. 
\end{definition}

In a first draft of this paper we call this by $q$-value, but in the final version a referee suggested to change by $s$-value. We can interpret the value $s$ as the greatest significance level for which at least one point of the closure of $\Theta_0$ lies inside the confidence region for $\theta$. When $F$ is continuous, a simple way of computing an $s$-value is to build high-confidence regions for $\theta$ that includes at least one point of the closure of $\Theta_0$ and gradually decreases the confidence until the $(1-s)\times 100\%$ confidence region border intercepts just the last point(s) of the closure of $\Theta_0$. The value $s$ is such that the $(1 - s + \delta)\times 100\%$ confidence region does not include any points of $\Theta_0$, for any $\delta>0$. Figure \ref{fig:def} illustrates some confidence regions  $\Lambda_{\alpha}$ for $\theta = (\theta_1, \theta_2)$ considering different values of $\alpha$. The dotted line is $\Lambda_{s_1}$, where $s_1$ is the $s$-value for testing $H_{01}: \theta_1 = 0$. The dot-dashed line  is $\Lambda_{s_2}$, where $s_2$ is the $s$-value for testing $H_{02}: \theta_2 = 0$. The dashed line  is $\Lambda_{s_3}$, where $s_3$ is the $s$-value for testing $H_{03}: \theta_1^2 + \theta_2^2 = 1$. In \cite{DuboisHH2004} is studied measures of confidence and confidence relations in a general fashion, here we shall show that our measure satisfies all confidence relations described by the authors. 

It must be said that $p$-values and confidence regions are naturally related when $H_0$ is simple and specifies the full vector of parameters. We will see that, in this precise case, $s$-values and $p$-values are the very same; on the other hand, if $H_0$ is simple and specifies just a partition of $\theta$, then $s$-values and $p$-values will be different. Also, when $H_0$ is composed (or specifies parameter curvatures) tests based on confidence regions are not readily defined. Our approach is a generalization of tests based on confidence regions under general composed null hypotheses. We shall see that this procedure has many interesting properties, is logically consistent and has a simple interpretation. We hope that these features would draw the attention of the statistical community for this new way to conduct tests of hypotheses.

\cite{Mauris} proposed a possibility measure based on confidence intervals to deal with fuzzy expression of uncertainty in measurement. This proposal is compatible with recommended guides on the expression of uncertainty. The authors consider ``identify each confidence interval of level $1-\alpha$, with each $\alpha$-cut of a fuzzy subset, which thus gathers the whole set of confidence intervals in its membership function'' \citep{Mauris}. The goals of the latter paper are different from those of this present paper, moreover the authors did not prove the properties of their proposed measure, which naturally depend on type of the adopted confidence intervals. Here, instead of considering confidence intervals, we consider regions of confidence and connects this general formulation to quantify the evidence yielded by data for or against the null hypothesis. In addition, we prove the properties that are essential for evidence measures considering this general formulation.

Observe that a large value of $s(\Theta_0)$ indicates that there exists at least one point in $\Theta_0$ that is near the maximum likelihood estimator, that is, data are not discrediting the null hypothesis $H_0$. Otherwise, a small value of $s(\Theta_0)$ means that all points of $\Theta_0$ are far from the maximum likelihood estimator, that is, data are discrediting the null hypothesis. The metric that says what is near or far from $\widehat{\theta}$ is the (asymptotic) distribution of $T_{\theta}$. These statements are readily seen by drawing confidence regions (or intervals) with different confidence levels, see Figure \ref{fig:def}. 

\cite{Bickel2012} developed a method based on the law of likelihood to quantify the weight of evidence for one hypothesis over another. Here, we proposed a classical possibility measure over $\Theta$ based on likelihood-based confidence regions, see Definition \ref{s-value}. Although these approaches are based on similar concepts, they capture different values from the data (we do not investigate this further in the present paper). As pointed out by a referee:  ``the proposed evidence measure relates to that proposed by \cite{Bickel2012} and \cite{Zhang2009} \emph{via} a monotone transformation determined by $F$. Because $F$ is fixed for a given model, there is an equivalence (up to a monotone transformation) between the two measures within each parametric model. However, since $F$ may change in different models (depending on the dimension of $\theta$, for instance), these two measures are not universally equivalent.'' This will be carefully investigate in further works. Another evidence measure that is a Bayesian competitor is the FBST (Full Bayesian Significance Test) proposed originally by \cite{PereiraStern}. See also an invariant version under reparametrizations in \cite{MadrugaPereiraStern} and we refer the reader to \cite{PereiraSternWechsler} for an extensive review of this latter method.

\section{Some important properties}\label{prop}

In this section we show some important properties of $s$-values that will be used to connect them with possibility measures and the abstract belief calculus (see Section \ref{ABCformalis}). First consider the following conditions:

\begin{enumerate}
\item[C1.] $\widehat{\theta}$ is an interior point of $\Theta$, 
\item[C2.] $\ell$ is strictly concave.
\end{enumerate}

Our first theorem states that item 1 of Definition \ref{Evid} holds for the proposed $s$-value.

\begin{theorem}\label{Item1} Let $s$ be an $s$-value and consider condition C1, then $s(\varnothing) = 0$ and $s(\Theta) = 1$. 
\end{theorem} 

\begin{proof} As $\widehat{\theta} \in \Theta$ (see condition C1), we have that $\{\alpha \in (0,1): \ \Lambda_\alpha \cap \Theta \neq \varnothing\}  = (0,1)$  and then $s(\Theta) = 1$. Also,  $\{\alpha \in (0,1): \ \Lambda_\alpha \cap \varnothing \neq \varnothing\}  = \varnothing$, then $s(\varnothing) = \max\{0, \sup(\varnothing)\} = 0$.   
\end{proof}

The following theorem completes the requirement for the $s$-value to be a measure of evidence.

\begin{theorem}\label{Nested-T}(Nested hypotheses) For a fixed data $X=x$ , let $H_{01}: \theta \in \Theta_{01}$ and $H_{02}: \theta \in \Theta_{02}$ be two null hypotheses such that $\Theta_{01} \subseteq \Theta_{02}$. Then, $s(\Theta_{01}) \leq s(\Theta_{02})$, where $s(\Theta_{01})$ and $s(\Theta_{02})$ are evidence measures for $H_{01}$ and $H_{02}$, respectively. 
\end{theorem}

\begin{proof} Observe that if $\Theta_{01} \subseteq \Theta_{02}$, then
$\Lambda_\alpha \cap \Theta_{01} \subseteq \Lambda_\alpha \cap \Theta_{02}$ and
$\{\alpha \in (0,1): \ \Lambda_\alpha \cap \Theta_{01} \neq \varnothing\} \subseteq \{\alpha \in (0,1): \ \Lambda_\alpha \cap  \Theta_{02} \neq \varnothing\}.$ We conclude that $s(\Theta_{01}) \leq s(\Theta_{02})$ for all $\Theta_{01} \subseteq \Theta_{02} \subseteq \Theta$.
\end{proof}

Other important feature of our proposal is its invariance under reparametrizations. As likelihood-based confidence regions are invariant under reparametrizations \citep[see][]{Schweder}, the $s$-value is also invariant.  Based on Theorem \ref{Nested-T} we can establish now an interesting result which is related to the Burden of Proof, namely, the evidence in favour of a composite hypothesis is the most favourable evidence in favour of its terms \cite[][]{Stern2003}.

\begin{lemma}\label{MFI}(Most Favourable Interpretation) Let $I$ be a countable or uncountable real subset, assume that $\Theta_0 = \bigcup_{i\in I} \Theta_{0i}$ is nonempty, then $s(\Theta_0) = \sup_{i \in I}\{s(\Theta_{0i})\}$.
\end{lemma}
\begin{proof} By Theorem \ref{Nested-T}, we know that $s(\Theta_0)\geq s(\Theta_{0i})$ for all $i\in I$, then $s(\Theta_0) \geq \sup_{i\in I}\{s(\Theta_{0i})\}$. To prove this lemma, we must show that $s(\Theta_0) \leq \sup_{i \in I}\{s(\Theta_{0i})\}$. 

Define $A(B) = \{\alpha \in (0,1): \Lambda_\alpha \cap B \neq \varnothing\}$ and note that $A(\Theta_0) \subseteq \bigcup_{i\in I}A(\Theta_{0i})$. Therefore, $\sup(A(\Theta_0)) \leq \sup\bigg(\bigcup_{i\in I}A(\Theta_{0i})\bigg)\Rightarrow \sup(A(\Theta_{0}))\leq \sup_{i \in I}\{\sup(A(\Theta_{0i}))\}$ and  $s(\Theta_0) = \sup_{i \in I}\{s(\Theta_{0i})\}$.
\end{proof}

Lemma \ref{MFI} states that $s$-values are possibility measures on $\Theta$, since they satisfy property in Lemma \ref{MFI} plus the conditions  in Definition \ref{Evid} \citep{DuboisPrade1990,Dubois2006}.  Next lemma presents an important result  (for strictly concave log-likelihood functions),  which allows us to connect $s$-values with $p$-values. 
  
\begin{lemma}\label{s-dist-lemma} Assume valid conditions C1 and C2. For a nonempty $\Theta_0$ and $F$ continuous and strictly increasing, the $s$-value can alternatively be defined as
\begin{equation}\label{s-dist}
s(\Theta_0) = \sup_{\theta \in \Theta_0}(1 - F(T_{\theta})) = 1 -F\big(T_{\Theta_0}\big),
\end{equation} where $T_\theta = 2(\ell(\widehat{\theta}) - \ell(\theta))$, $T_{\Theta_0} = 2(\ell(\widehat{\theta}) - \ell(\widehat{\theta}_0))$ and $\widehat{\theta}_0 = \arg \sup_{\theta \in \Theta_0} \ell(\theta)$. 
\end{lemma}

\begin{proof}  If $\Theta_0$ is nonempty, there exists $\alpha \in (0,1)$ such that $\Lambda_\alpha \cap \Theta_0 \neq \varnothing$ and the $s$-value is just $s(\Theta_0) = \sup\{\alpha \in (0,1): \ \Lambda_\alpha \cap \Theta_{0} \neq \varnothing\}$.
 Notice that, as $F$ is strictly increasing, we have that
\begin{align*}
\Lambda_\alpha \cap \Theta_{0} & \equiv \{\theta \in \Theta_0: T_\theta \leq F_\alpha\}\\
& \equiv  \big\{\theta \in \Theta_0: F(T_\theta ) \leq 1-\alpha\big\}\\
&\equiv \big\{\theta \in \Theta_0:  1 - F(T_\theta ) \geq \alpha\big\}.
\end{align*}
As $\ell$ is strictly concave, then for all $0<\alpha\leq s\leq1$ we have $\Lambda_s \cap \Theta_0 \subseteq \Lambda_{\alpha}\cap \Theta_0$ and  \[\sup\{\alpha \in (0,1): \ \Lambda_\alpha \cap \Theta_{0} \neq \varnothing\}= \sup_{\theta \in \Theta_0} \{1 -F(T_\theta)\}\] and as $F$ is continuous and strictly increasing
\[\sup_{\theta \in \Theta_0} \{1 -F\big(T_\theta)\big)\} =1 -F\big(T_{\Theta_0}),  \]
 where  $T_\theta = 2(\ell(\widehat{\theta}) - \ell(\theta))$, $T_{\Theta_0} = 2(\ell(\widehat{\theta}) - \ell(\widehat{\theta}_0))$ and $\widehat{\theta}_0 = \arg\sup_{\theta \in \Theta_0} \ell(\theta)$. 
\end{proof}
Lemma \ref{s-dist-lemma} basically states that our proposed measure is isomorphic to the likelihood statistic when $F$ is continuous and strictly increasing and the log-likelihood function $\ell$ is strictly concave. This lemma connects our proposal with the work of  \cite{Dubois1997} and then, if the assumptions of this lemma hold, all results derived by these authors are also valid for our proposal.

The value of $\widehat{\theta}_0 $ can be seen as the point of $\Theta_0$ which is in the boundary of $(1-s)\times 100\%$ confidence region for $\theta$. Notice that if $F$ is a non-decreasing function then we just can claim that \[\{\theta \in \Theta_0: T_\theta \leq F_\alpha\} \subseteq  \big\{\theta \in \Theta_0: F(T_\theta) \leq 1-\alpha\big\},\] the converse inclusion may not be valid. Typically, $F$ can be approximated to a quisquare distribution with $k$ degrees of freedom, where $k$ is the dimension of $\Theta$ (this is a continuous and strictly increasing function). Based upon this alternative version we can directly compare $s$-values with $p$-values. In addition, it is possible to derive the distribution of $s$.

Notice that, for one-sided null hypotheses $H_0$ and monotonic likelihood ratio, the corresponding $p$-value would be easily computed by 
\[p = P_{\Theta_0}(T_{\Theta_0} > t_{\Theta_0}) = 1 - F_{H_0}(t_{\Theta_0}),\] 
 where $F_{H_0}(t_{\Theta_0}) = P_{\Theta_0}(T_{\Theta_0}\leq t_{\Theta_0}) $ is the (asymptotic) cumulative distribution of $T_{\Theta_0}$ that depends on $H_0$ \citep[see][]{Mudholkar} and $t_{\Theta_0}$ is the observed value of $T_{\Theta_0}$. Now, by Equation (\ref{s-dist}), we find a duality between $s$-values and $p$-values which is self-evident from Lemma \ref{s-dist-lemma}.

\begin{lemma}\label{l-rel} Consider valid conditions C1 and C2. If $F$ and $F_{H_0}$ are continuous and strictly increasing functions and $H_0$ is an one-sided hypothesis, then the following equalities hold 
\begin{equation}\label{relation}p = 1 - F_{H_0}(F^{-1}(1 - s)) \quad \mbox{and} \quad s = 1 - F(F^{-1}_{H_0}(1 - p)). \end{equation}
\end{lemma}

When $H_0$ is a two-sided hypothesis, then the relation between the $s$-value and the $p$-value is different, but in this paper we do not investigate this further. Also, optimal $p$-values when the null hypothesis is two-sided is proposed by \cite{Mudholkar},  a connection with $s$-values might be studied in further works. Note that,  under general regularity conditions on the likelihood function and considering that $\Theta_0$ is a smooth semi-algebraic subset of $\Theta$, all asymptotics for the $s$-value can be derived by using the duality relation presented above in Equation \ref{relation}.

As for the case where $T_{\Theta_0}$ is not (asymptotically) ancillary to $\mathcal{P}_T^0$, the most conservative p-value computed by using the likelihood ratio statistic
\[p(\Theta_0)= \sup_{\theta \in \Theta_0} P_{\theta, \Theta_0}(T_{\Theta_0}>t_{\Theta_0}) = \sup_{\theta \in \Theta_0} (1 - P_{\theta, \Theta_0}(T_{\Theta_0} \leq t_{\Theta_0})) = 1 -  F_{H_0}^*(t_{\Theta_0})\] where \[F_{H_0}^*(t_{\Theta_0}) = \inf_{\theta \in \Theta_0} P_{\theta,\Theta_0}(T_{\Theta_0} \leq t_{\Theta_0})\] is a cumulative function the depends on $H_0$ and if it is continuous and strictly increasing Lemma 3.3 is applicable.

 In usual frequentist significance tests, the error probability of type I characterizes the proportion of cases in which a null hypothesis $H_0$ would be rejected when it is true in a hypothetical long-run of repeated sampling. On the one hand, as a $p$-value usually has uniform distribution under $H_0$, the probability to obtain a $p$-value smaller than $\alpha$ is $\alpha$. On the other hand, we can only guarantee uniform distribution for the evidence value $s$ under the simple null hypothesis $\Theta_0 = \{\theta_0\}$, which specifies all parameters involved in the model. Since, it can be readily seen that if $\Theta_0 = \{\theta_0\}$, then $F \equiv F_{H_0}$ (an asymptotic quisquare with $k$ degrees of freedom) and therefore $s \sim U(0,1)$ (at least asymptotically). However, if $\Theta_0$ has dimension less than $k$, e.g., under curvature of parameters, the distribution $F_{H_0}$ would differ from $F$. Notice that the threshold value $\alpha$ adopted for $p$-values is not valid for $s$-values, the actual threshold for $s$-values should be computed using relation (\ref{relation}), i.e., $\alpha_{H_0} = 1 - F(F_{H_0}^{-1}(1-\alpha))$ would be the new cut-off. Of course, if the decision was based on this actual threshold, the same logical contradictions would arise.

\cite{PereiraSternWechsler} left a challenge to the reader, namely, to obtain the one-to-one relationship between the evidence value computed via FBST ($e$-value) and $p$-values. \cite{Diniz2012} showed that asymptotically the answer is given in Lemma \ref{l-rel} replacing the $s$-value with $e$-value, therefore, $s$-values and $e$-values are asymptotically equivalent. \cite{Polansky} proposed an observed confidence approach for testing hypotheses, however this approach differs from ours because, as stated on its Section 2.2 (The General Case), the proposed measure must satisfy the probability axioms. As we have seen, the $s$-value does not satisfy the probability axioms, instead it satisfies the possibility axioms.  Also, \cite{Polansky} seems to build confidence regions around the null set $\Theta_0$, which is a very different approach from the one we are proposing in this paper.

\section{$S$-value as an objective state of belief}\label{ABCformalis}

In this section we analyze the definition of $s$-values under the light of Abstract Belief Calculus (ABC) proposed by \cite{DarwicheGinsberg}. ABC is a symbolic generalization of probabilities. Probability is a function defined over a family of subsets (known as $\sigma$-field) of a main nonempty set $\Omega$  to the interval $[0,1]$. The additivity is the main characteristic of this function of subsets, that is, if $A$ and $B$ are disjoint measurable subsets, then the probability of the union $A \cup B$ is the sum of their respective probabilities. Basically, all theorems of probability calculus require this additive property. \cite{Cox1946} derived this sum rule from a set of more fundamental axioms \citep[see also][]{Jaynes1957, Aczel2004}, which tries to consider the following assertions
\begin{quote}
``... the less likely is an event to occur the more likely it is not to occur. The occurrence of both of two events will not be more likely and will generally be less likely than the occurrence of the less likely of the two. But the occurrence of at least one of the events is not less likely and is generally more likely than the occurrence of either'' \citep{Cox1946}.
\end{quote} 
The quotation above is alleged to be a fundamental part of any coherent reasoning and, based on it, some scholars claim that degrees of belief should be manipulated according to the laws of probability theory \citep{Cox1946, Jaynes1957, Ariel}.  The word ``likely'' could be replaced by ``probable'', ``possible'', ``plausible'' or any other that represents a measure for our belief or (un)certainty under  limited knowledge. \cite{DuboisPrade2001} said that, under a limited knowledge, 
\begin{quote}
``one agent that does not believe in a proposition does NOT imply the (s)he believes in its negation''
\end{quote} and also that 
\begin{quote}``uncertainty in propositional logic is \emph{ternary} and not binary: either a proposition is believed, or its negation is believed, or neither of them are believed''.
\end{quote}
Therefore, under a limited knowledge, the claim that ``the less possible is an event to occur the more possible it is not to occur'' is too restrictive to be of universal applicability in general beliefs. 

The Cox's demonstration is made through associativity functional equations (let $G: \mathbb{R}^2 \to \mathbb{R}$ be a real function with two real arguments, then $G(x,G(y,z)) = G(G(x,y),z)$ is the associativity functional equation) that represent the second sentence of the above quotation, the involved function ($G$) is considered continuous and strictly increasing in both its arguments. However, when this function is non-decreasing we also have a coherent reasoning \citep[see][]{DarwicheGinsberg} and other than additive rules emerge from this functional equation, such as minimum (maximum) as demonstrated by \cite{Marichal2000}. Therefore, probability is not the unique coherent way of dealing with uncertainties as usually thought and spread among some scholars. Moreover, \cite{DuboisPrade2001} clarify that probability theory is not a faithful representation of incomplete knowledge in the sense of classical logic as usual considered. 

In our case, the set that we want to define an objective state of belief is the parameter space $\Theta$. Here, the word ``objective'' means that no prior distributions over $\Theta$ are specified. Naturally there exists some level of subjective knowledge in the choice of models, parameter space and so on, these sources of subjectivity will not be discussed further. It is widely known that the frequentist school regards no probability distributions over the subsets of $\Theta$ and that probability distributions are only assigned for observable randomized events. Here, we show that the proposed $s$-value can prescribe an objective state of belief over the subsets of $\Theta$ without assigning any prior subjective probability distributions over the subsets of $\Theta$. It is quite obvious that an $s$-value is not a probability measure, since it is not additive. However, as we shall see in this section, $s$-values are indeed abstract states of belief.

Abstract belief calculus (ABC) is built considering more basic axioms than regarded in Cox derivation and, as a consequence, the additive property must be  generalized to a summation operator $\oplus$ such that the usual sum rule is a simple particular case (this theory also deals with propositions instead of sets, but here we consider that propositions are represented explicitly as sets). For the sake of completeness, we expose the main components of the ABC theory in what follows. Firstly, let $\mathcal{L}$ be a family of subsets $\Theta$ closed under unions, intersections and complements. The ABC starts defining a function $\Phi: \mathcal{L}\to \mathcal{S}$ called support function, and for each $A \in \mathcal{L}$, $\Phi(A)$ is called the support value of $A$. Then, in order to define the properties of this support function, a partial support structure $\langle \mathcal{S}, \oplus\rangle$ is defined such that the summation support  $\oplus: \mathcal{S} \times \mathcal{S} \to \mathcal{S}$ satisfies the following properties: 
\begin{itemize}

\item Symmetry: $a \oplus b = b \oplus a$ for every $a,b \in \mathcal{S}$
\item Associativity: $(a \oplus b) \oplus c = a \oplus(b \oplus c)$, for every $a,b,c \in \mathcal{S}$
\item Convexity: For every $a,b,c \in \mathcal{S}$, such that $(a\oplus b) \oplus c = a$, then also $a \oplus b = a$
\item Zero element: There exists a unique $\bm{0}\in \mathcal{S}$ such that $a \oplus \bm{0} = a$ for all $a \in \mathcal{S}$
\item Unit element: There exists a unique element $\bm{1} \in \mathcal{S}$, where $\bm{1} \neq \bm{0}$, such that for each $a \in \mathcal{S}$, there exists $b \in \mathcal{S}$ such that $a \oplus b = \bm{1}$
\end{itemize}

Then, the properties of $\Phi$ are the following
\begin{enumerate}
\item For $A, B \in \mathcal{L}$, such that $A \subseteq B$ and $B \subseteq A$, then $\Phi(A) = \Phi(B)$
\item For $A, B \in \mathcal{L}$, such that $A \cap B = \emptyset$,  then $\Phi(A \cup B) = \Phi(A)\oplus \Phi(B)$
\item For $A, B, C \in \mathcal{L}$, such that $A \subseteq B \subseteq C$ and $\Phi(A) = \Phi(C)$, then $\Phi(A) = \Phi(B)$
\item $\Phi(\varnothing) = \bm{0}$ and $\Phi(\Theta) = \bm{1}$
\end{enumerate}

Backing to our proposal and taking $\mathcal{S} = [0,1]$, $\oplus \equiv \sup$,  $\bm{0} = 0$, $\bm{1} = 1$ and $\Phi \equiv s$ we see, by Theorems \ref{Item1} and \ref{Nested-T}, that the $s$-value satisfies all the properties above. Therefore, for our problem, we identify that our $s$-value is acting precisely as a support function $\Phi$ in the ABC formalism. It is noteworthy that the support value of $A$ does not determine the support value of $A^c$ (the complement of $A$). This determination happens in the probability calculus since, when $\mathcal{S} = [0,1]$ and $\oplus \equiv +$, the probabilities of sets that form a partition of the total space must sum up to one. Then, for abstract support functions, \cite{DarwicheGinsberg} defined the degree of belief function $\ddot{\Phi}: \mathcal{L} \to \mathcal{S}\times \mathcal{S}$, such that $\ddot{\Phi}(A) = \langle \Phi(A), \Phi(A^c)\rangle$. The value of $\ddot{\Phi}(A)$ is said to be the degree of belief of $A$. Naturally, for the probability calculus this function is vacuous, since if $\mathcal{S}=[0,1]$, $\oplus \equiv +$ and $\Phi$ is a probability function then for any measurable subset $A$, $\ddot{\Phi}(A) = \langle \Phi(A), 1 - \Phi(A)\rangle $.

Now, let $a, b \in \mathcal{S}$, then if there exists  $c \in \mathcal{S}$ such that $a \oplus c = b$ we say that the support value $a$ is no greater than the support value $b$ and we use the notation $a \preceq_\oplus b$, the symbol $\preceq_\oplus$ is called  as support order.  \cite{DarwicheGinsberg} showed that $\preceq_\oplus$ is a partial order under which $\bm{0}$ is minimal  and $\bm{1}$ is maximal. Also, let $\ddot{\Phi}(A) = \langle a_1, a_2 \rangle$ and $\ddot{\Phi}(B) = \langle b_1, b_2\rangle$ be degrees of beliefs of $A$ and $B$, respectively. We say that the degree of belief of $A$ is no greater than the degree of belief of $B$ if $a_1 \preceq_\oplus b_1$ and $b_2 \preceq_\oplus b_1$, this is represented by $\ddot{\Phi}(A)\sqsubseteq_\oplus \ddot{\Phi}(B)$. \cite{DarwicheGinsberg} also showed that $\sqsubseteq_\oplus$ is a partial order under which $\langle\bm{0},\bm{1}\rangle$ is minimal and  $\langle\bm{1}, \bm{0}\rangle$ is maximal. In this coherent framework, $\ddot{\Phi}(A) = \langle \bm{1}, \bm{1}\rangle$ is fully possible without having any contradictions (of course that in probability measures this cannot happen). The authors clarified this in terms of propositions, see the quotation below:
\begin{quote}
``A sentence is rejected precisely when it is supported to degree $\bm{0}$. And a sentence is accepted only if it is supported to degree $\bm{1}$. But if the sentence is supported to degree $\bm{1}$, it is not necessarily accepted. For example, when degrees of support are $\mathcal{S} = \{\mbox{\emph{possible}, \emph{impossible}}\}$, a sentence and its negation could be \emph{possible}. Here, both the sentence and its negation are supported to degree \emph{possible}, but neither is accepted.''
\end{quote}

As our $s$-value is a support function for the subsets of $\Theta$, we may used this measure to state degrees of belief for the subsets of $\Theta$. Notice that we cannot do this with the usual concept of $p$-value by the following. If $A \subset B \subset \Theta$, then $A $ e $A^c\cap B$ are disjoint sets, therefore as $B = A \cup (A^c\cap B)$ we have by Property 2 that $\Phi(B) = \Phi(A)\oplus \Phi(A^c\cap B)$ and then $\Phi(A) \preceq_\oplus \Phi(B)$, that is, if $A$ is a subset of $B$ the support value $\Phi(A)$ is no greater than the support value $\Phi(B)$. By our Examples \ref{Ex-normal1} and \ref{reg} we see that $p$-values do not satisfy this requirement. As aforementioned, $s$-values are not probabilities measures on the subsets of $\Theta$ and this is not a weakness as some may argue. For instance, for subsets of $\Theta$ with dimensions smaller than $\Theta$, the best  measure of evidence that probability measures can provide is zero. We remark that $p$-values and the Bayesian $e$-values \citep{PereiraStern} are not probability measures on $\Theta$, moreover, usually defined $p$-values cannot even be included in the ABC formalism to establish objective states of belief on the subsets of $\Theta$. Hence, $s$-values can be used to fill this gap.

Consider the null hypothesis $H_0: \theta \in \Theta_0$, by Condition C1, we know that $\widehat{\theta} \in \Theta_0$ or $\widehat{\theta} \in \Theta_0^c$, then $s(\Theta_0) = 1$ or $s(\Theta^c)=1$. As the $s$-value is a measure of support and $\langle\bm{0},\bm{1}\rangle$ is minimal and $\langle\bm{1}, \bm{0}\rangle$ is maximal, we can readily reject $H$ provided that $\ddot{\Phi}(\Theta_0) = \langle s(\Theta_0),s(\Theta_0^c)\rangle= \langle 0,1\rangle$ and readily accept $H$ provided that  $\ddot{\Phi}(\Theta_0) =  \langle s(\Theta_0),s(\Theta_0^c)\rangle = \langle 1,0\rangle$. In these two cases we have complete knowledge. When $\ddot{\Phi}(\Theta_0) = \langle s(\Theta_0),s(\Theta_0^c)\rangle= \langle 1,1\rangle$ we have complete ignorance regarding this specific hypothesis and cannot either accept or reject, then we must perform other experiment for gathering more information. Typically, we have intermediate states of knowledge, namely: (1) $\ddot \Phi(\Theta_0) = \langle a,1\rangle$ or (2) $\ddot \Phi(\Theta_0) = \langle 1,b\rangle$, where $a,b \in [0,1]$. In  the first case, we say that there is evidence against $H_0$ if $a$ is sufficiently small (i.e., $a < a_c$ for some critical value $a_c$) and we say that the decision is unknown whenever $a$ is not sufficiently small. In the second case, we say that there is evidence in favor of $H_0$ if $b$ is sufficiently small (i.e., $b < b_c$ for some critical value $b_c$) and we say that the decision is unknown whenever $b$ is not sufficient small. The problem now is to define what is ``sufficiently small'' to perform a decision. The decision rule may be derived through loss functions or other procedure, we will study this issue in future works. Whatever the chosen procedure, it should respect the minimal ($\ddot \Phi(\Theta_0) = \langle 0,1\rangle$), maximal ($\ddot \Phi(\Theta_0) = \langle 1,0\rangle$) and inconclusive  ($\ddot \Phi(\Theta_0) = \langle 1,1\rangle$) features of possibility measures.

\section{Examples}\label{Exam}

In this section we apply our proposal to the Examples \ref{Ex-normal1} and \ref{reg} and we also consider an example for the Hardy-Weinberg equilibrium hypothesis. 

\begin{example}\label{Ex-normal2} Consider Example \ref{Ex-normal1}, after a straightforward computation we find \[T_\theta =2(\ell(\widehat{\theta}) - \ell(\theta)) = n(\bar{X} - \mu)^\top(\bar{X} - \mu),\] where $\bar{X} = (\bar{X}_1, \bar{X}_2)^\top$ and $\mu = (\mu_1, \mu_2)^\top$ and $F_\alpha$ is the $\alpha$ quantile from a quisquare distribution with two degrees of freedom. Then, $s$-values for $H_{01}: \mu_1=\mu_2 = 0$ and $H_{02}: \mu_1 = \mu_2$ are respectively $s_1 =  0.104$ and $s_2 =  0.105$. Note that the curve $\mu_1 = \mu_2$ intercepts $\Lambda_{s_2}$ at $(-0.01, -0.01)$. As expected for this case, $s_1 < s_2$, since $\Theta_{01} \subset \Theta_{02}$, in addition, $s_1$ and $s_2$ are near each other because the variables are independent. If the variables were correlated, those $s$-values would differ drastically (being $s_2$ always greater than $s_1$). 
\end{example}

\begin{example}\label{reg2} Consider Example \ref{reg}, the maximum likelihood estimates for $b_1$ and $b_2$ are respectively $\widehat{b}_1 = 0.1966$ and $\widehat{b}_2 = -0.1821$. Here, we find that \[T_\theta=2(\ell(\widehat{\theta}) - \ell(\theta)) = (\widehat{b} - b)^\top (x^\top x)(\widehat{b} - b).\] Then, $F_\alpha$ is the $\alpha$ quantile from a quisquare distribution with two degrees of freedom and the $s$-values for $H_{01}: b_1 = 0$, $H_{02}: b_2 = 0$ and $H_{03}: (b_1,b_2) = (0,0)$ are respectively $s_1 = 0.107$, $s_2 = 0.135$ and $s_3 = 0.101$. Therefore, as expected by the logical reasoning $s_1> s_3$ and $s_2 >s_3$.
\end{example}

We also compare our results with the FBST approach considering a trinomial distribution and the Hardy-Weinberg equilibrium hypothesis. 

\begin{example}\label{trinomial}
 Consider that we observe a vector of three values $x_1, x_2, x_3$,  in which the likelihood function is proportional to $\theta_1^{x_1}\theta_2^{x_2}\theta_3^{x_3}$, where $x_1+x_2+x_3 = n$ and the parameter space is $\Theta = \{\theta \in (0,1)^3: \ \theta_1+\theta_2 + \theta_3 = 1\}$. Here, we use the same settings described in Section  4.3 by \cite{PereiraStern}, that is, the null hypothesis is $\Theta_0 = \bigg\{\theta \in \Theta: \ \theta_3 = (1 - \sqrt{\theta_1})^2\bigg\}$ and $n=20$. 
\end{example}

Table \ref{tab:1} presents the $s$-values for all values of $x_1$ and $x_3$. The last two columns were taken from Table 2 of \cite{PereiraStern}. It should be said that we computed the $s$-values by using Definition \ref{s-value} instead of Relation (\ref{relation}), because the $p$-values were presented with two decimal places in \cite{PereiraStern} and this can induce distorted $s$-values. As it was seen, our proposal yields similar results to the FBST approach.

\section{Discussion and final remarks}\label{Disc}

\cite{BergerSellke} compare $p$-values with posterior probabilities (by using objective prior distributions) and find differences by an order of magnitude (when testing a normal mean, data may produce a $p$-value of $0.05$ and posterior probability of the null hypothesis of at least $0.30$). For an extensive review on  the relation between p-values and posterior probabilities, the reader is referred to \cite{Ghosh}. As opposed to posterior distributions, $p$-values do not hold the requirement of evidence measures, but one can also conclude that posterior probabilities cannot be used to reflect probabilities in a hypothetical long-run of repeated sampling. Also, posterior probabilities cannot provide a measure of evidence different from zero under sharp hypotheses (when the dimension of the null parameter space is smaller than the full parameter space). A Bayesian procedure that provides a positive evidence measure under sharp hypotheses is the FBST. In this context, $s$-values are directly comparable with the evidence measures of the FBST approach.  This latter procedure needs numerical integrations and maximizations, which may be difficult to be attained  for high dimensional problems. As we studied in previous sections, our procedure produces similar results to the FBST and can be readily used as a classical alternative (when the user does not want to specify prior distributions). Moreover, if one has a $p$-value computed via likelihood ratio statistic, then Relation (\ref{relation}) may be applied to compute the respective $s$-value without any further computational procedures (maximizations and integrations)  and, also, this relation allows to derive the $s$-value distribution (if desired).  Also, we should mention that $s$-values, when computed using the asymptotic distribution of $T_{\theta_0}$, do respect the famous likelihood principle, but we should say that it is not the main concern here, it is just a property of our approach. Naturally, if the exact distribution of $T_{\theta_0}$ is adopted, the likelihood principle may be violated.

When the null hypothesis is simple and specifies the full vector of parameters, say $\Theta_0 = \{\theta_0\}$, the proposed $s$-values are, in general, $p$-values \citep{Schweder}. Otherwise, $s$-values cannot be interpreted as $p$-values, instead, they must be treated as measures of evidence for null hypotheses. As aforementioned, when treated as evidence measures, $p$-values have some internal undesirable features (in some cases, for nested hypotheses $H_{01}$ and $H_{02}$, where $H_{01}$ is nested within $H_{02}$, $p$-values might give more evidence against $H_{02}$ than $H_{01}$). On the other hand, $p$-values respect the repeated sampling principle, that is, in the long-run average actual error of rejecting a true hypothesis is not greater than the reported error. In other words, as $p$-values have uniform distribution under the null hypothesis, the frequency of observing $p$-values smaller than $\alpha$ is $\alpha$. This is an external desirable aspect, since this allows us to verify model assumptions and adequacy, among many other things. The proposed $s$-values overcome that internal undesirable aspect of $p$-values, but the problem now is how to evaluate a critical value to establish a decision rule for a hypothesis based on $s$-values (this decision should respect the rules of possibility measures). If we want to respect the repeated sampling principle, based on Lemma \ref{l-rel}, we see that a critical value for $s$ depends on $H_0$. To see that, let $\alpha$ be the chosen critical value for the computed $p$-value, then it can be ``corrected'' to $\alpha_{H_0}' = 1 - F(F_{H_0}^{-1}(1-\alpha))$ for the respective $s$-value. This threshold value $\alpha_{H_0}'$ will respect the repeated sampling principle if and only if it varies with $H_0$. If we adopt this ``corrected'' critical value we will have the same internal undesirable features of  $p$-values. We must rely on other principles to compute the threshold value for our $s$-value, maybe based on loss functions. These loss functions may incorporate the scientific importance of a hypothesis to elaborate a reasonable critical value (this issue will be discussed in future work). It is well known that statistical significance is not the same as scientific significance, for a further discussion we refer the reader to \cite{Cox1977}. Naturally, we could also employ loss functions on $p$-values to find a threshold, however, the internal undesirable features of $p$-values will certainly bring problems to implement this without any logical conflicts.

There are many open issues that need more attention regarding $s$-values. Next we provide a list of open problems that we did not deal with in this article, but will be subject of our future research.  

\begin{enumerate}

\item To give a rigorous mathematical treatment when the log-likelihood function $\ell$ is not strictly concave. 

\item To derive a computational procedure to find $s$-values and their distribution for (semi)algebraic subsets $\Theta_0$ and not strictly concave $\ell$.

\item To compare theoretical properties of $s$-values by using other types of confidence regions. Monte Carlo simulations may be required.

\item To compare $s$-values with evidence values (e-value) computed via FBST \citep{PereiraStern} and other procedures such as the posterior Bayes factor \citep{Aitkin} in a variety of models by using actual data.

\item To derive a criterion to advise one out of three decisions ``acceptance'', ``rejection'' or ``undecidable'' of a null hypothesis without having any types of conflict.

\end{enumerate}

We end this paper by saying that we are not advocating a replacement of $p$-values by $s$-values. Instead, we just recommend $s$-values as additional measures to assist data analysis.  

\section*{Acknowledgements}

I gratefully acknowledge partial financial support from FAPESP. I also wish to thank Nat\'alia Oliveira Vargas and Silva for valuable suggestions on the writing of this manuscript, Corey Yanofsky for bringing to my attention the  Bickel's report \citep{Bickel2012} and Jonatas Eduardo Cesar for many valuable discussions on similar topics. This paper is dedicated to Professor Carlos Alberto de Bragan\c ca Pereira (Carlinhos) who motivates  his students and colleagues to think on the foundations of probability and statistics. He is head at the Bayesian research group at University of S\~ao Paulo and has made various contributions to the foundations of statistics. I also would like to thank three anonymous referees and the associate editor for their helpful comments and suggestions that led to an improved version of this paper.

\appendix

\begin{figure}[!htp]
\begin{center}
\includegraphics[angle = 270, scale = 0.5]{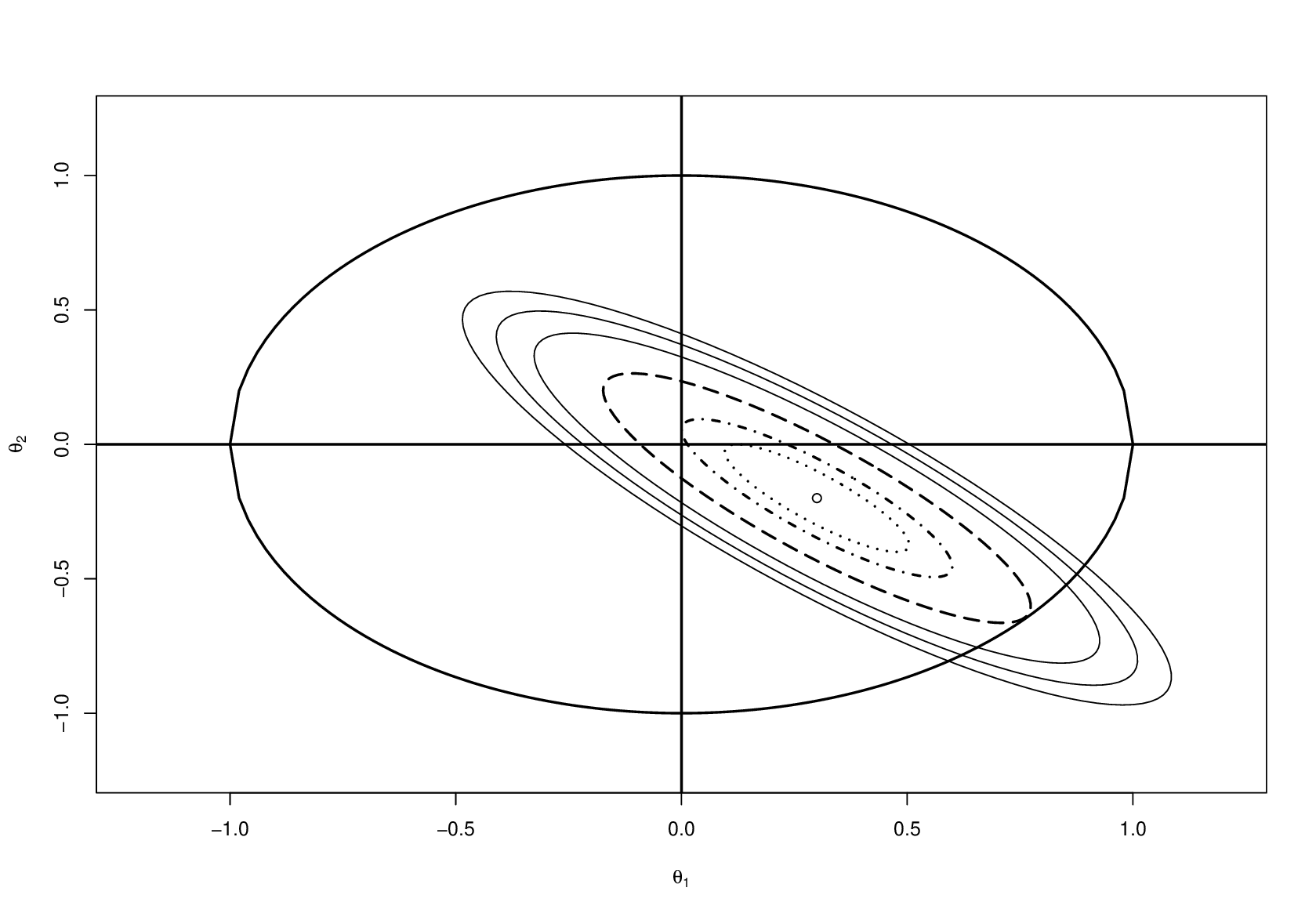}
\end{center}\caption{Borders of confidence regions $\Lambda_{\alpha}$, for different values of $\alpha$. The dotted line is $\Lambda_{s_1}$, where $s_1$ is the $s$-value for testing $H_{01}: \theta_1 = 0$. The dot-dashed line  is $\Lambda_{s_2}$, where $s_2$ is the $s$-value for testing $H_{02}: \theta_2 = 0$. The dashed line  is $\Lambda_{s_3}$, where $s_3$ is the $s$-value for testing $H_{03}: \theta_1^2 + \theta_2^2 = 1$. }
\label{fig:def}
\end{figure}

\begin{table}[!htp]\label{tab:1}
\centering
\begin{center}\caption{Tests of Hardy-Weinberg equilibrium}
\begin{tabular}{ccccc}\hline
$x_1$&$x_3$&$s$-value & $e$-value  & $p$-value \\  
     &     &          &  (FBST)    &     \\  \hline
1&2&0.00 &  0.01 &0.00 \\
1&3&0.02 &0.01 &0.01 \\
1&4&0.04&0.04 &0.02 \\
1&5&0.10&0.09 &0.04 \\
1&6&0.20&0.18 &0.08 \\
1&7&0.33&0.31 &0.15 \\
1&8&0.50&0.48 &0.26 \\
1&9&0.68&0.66 &0.39 \\
1&10&0.84&0.83 &0.57 \\
1&11&0.95&0.95 &0.77 \\
1&12& 1.00&1.00 &0.99 \\
1&13& 0.96&0.96 &0.78 \\
1&14& 0.85&0.84 &0.55 \\
1&15&0.68&0.66 &0.33 \\
1&16&0.48&0.47 &0.16 \\
1&17&0.29&0.27 &0.05 \\
1&18&0.13&0.12 &0.00 \\
5&0&0.01&0.02 &0.01 \\
5&1&0.10&0.09 &0.04 \\
5&2&0.32 &0.29 &0.14 \\
5&3& 0.63&0.61 &0.34 \\
5&4& 0.90&0.89 &0.65 \\
5&5&1.00      &  1.00            & 1.00      \\
5&6&0.91      &  0.90            & 0.66      \\
5&7&0.69      &  0.66            & 0.39      \\
5&8&0.44      &  0.40            & 0.20      \\
5&9&0.24      &  0.21            & 0.09      \\
5&10&0.11      &  0.09            & 0.04     \\
9&0&0.12      &  0.21            & 0.09      \\
9&1&0.68      &  0.66            & 0.39      \\
9&2&0.99      &  0.99            & 0.91      \\
9&3&0.87      &  0.86            & 0.59      \\
9&4&0.53      &  0.49            &0.26       \\
9&5&0.24      &  0.21            &0.09       \\
9&6&0.08      &  0.06            &0.03       \\
9&7&0.02      &  0.01            &0.01       \\\hline
\end{tabular}
\end{center}
\end{table}
\end{document}